\documentclass[12pt,preprint,twoside]{article} 
 \usepackage{dsfont}
 \usepackage{amsmath}
 \usepackage{amsthm}
 \usepackage{txfonts,bm}
 \usepackage[sc,bf]{caption}
 \usepackage{latexsym}
 \usepackage{rotating}
 \usepackage{fancyhdr}
 \usepackage{fancyheadings}
 \RequirePackage[OT1]{fontenc}
 \RequirePackage{amsthm,amsmath}
 \RequirePackage[numbers]{natbib}
 \usepackage[colorlinks,citecolor=blue,urlcolor=blue,linkcolor=blue]{hyperref}

 \usepackage{titlesec}
 \titleformat{\section}{\large\bfseries}{\thesection}{1em}{}
 \titleformat{\subsection}{\normalsize\bfseries}{\thesubsection}{1em}{}
 \titleformat{\subsubsection}{\normalsize\it}{{\rm \thesubsubsection}}{1em}{\vspace{-.7em}}

 \newtheoremstyle{mytheo}
                {1ex}
                {0.4ex}
                {\rm}
                {}
                {\bfseries\scshape }
                {.}
                {0.5ex}
                {}

 \theoremstyle{mytheo}

 \newtheorem{theo}{Theorem}[section]

 \newtheorem{cor}{Corollary}[section]
 \newtheorem{defi}{Definition}[section]
 \newtheorem{exam}{Example}[section]
 \newtheorem{lem}{Lemma}[section]
 \newtheorem{prop}{Proposition}[section]
 \newtheorem{rem}{Remark}[section]

 \numberwithin{equation}{section}

 \newcommand{\be}{\begin{equation}}
 \newcommand{\ee}{\end{equation}}
 \newcommand{\E}{\operatorname{\mathds{E}}}
 \renewcommand{\Pr}{\operatorname{\mathds{P}}}

  \newcommand{\RR}{\mathds{R}}

 \title{\Large\bf Sequences of
 expected
 record values\textcolor[rgb]{0,0,1}{\footnote{Work partially supported
 by the National and Kapodistrian University of
 Athens' Research fund under Grant
 70/4/5637.}}\vspace*{-.6em}}

 \author{\large
 Nickos
 Papadatos
 {\footnote{
 {e-mail:}\
 \textcolor[rgb]{0.98,0.00,0.00}{npapadat@math.uoa.gr}, {url:}\
 \textcolor[rgb]{0.98,0.00,0.00}{users.uoa.gr/$\sim$npapadat/}}}}
 
 \vspace{-1em}

 \date{\small\it
 \begin{tabular}{r@{\hspace{0ex}}l}
 & National and Kapodistrian University of Athens, Department of Mathematics,
 \\
 [-.3ex]
 &
 Section of Statistics and Operational
 Research,
 Panepistemiopolis, 157 84 Athens, Greece.
 \end{tabular}
 }
 \vspace{-1em}

\begin{document}

 \maketitle
 \vspace*{-2em}

 \thispagestyle{empty}

 \begin{abstract}
 \noindent
 We investigate conditions in order to decide whether a given
 sequence of real numbers represents expected record values
 arising from an independent, identically distributed, sequence
 of random variables.
 The main result provides a necessary and sufficient condition,
 relating any expected record sequence with the Stieltjes moment problem.
 The results are proved by means of a useful transformation on
 random variables. Some properties of this mapping, and its inverse,
 are discussed in detail, and, under mild conditions,
 an explicit inversion formula for the
 random variable that admits a given expected record sequence
 is obtained.
 \end{abstract}
 {\footnotesize {\it MSC}:  Primary 60E05; 62G30; Secondary 44A60.
 \newline
 {\it Key words and phrases}: characterizations; expected record values;
 \vspace*{-.7ex}
 Stieltjes moment problem; transformation of random variables; inversion formula.
 }
 \vspace*{-1em}

 \section{Introduction}
 \setcounter{equation}{0}
 \label{sec.1}

 Let $X$ be a random variable (r.v.)\
 with distribution function (d.f.)\ $F$,
 and suppose that $X_1,X_2,\ldots$
 is an independent, identically distributed (i.i.d.)\
 sequence from $F$.
 The usual record times, $T_n$, and (upper) record values,
 $R_n$, corresponding
 to the i.i.d.\ sequence $X_1,X_2,\ldots$, are defined by
 $T_1=1$, $R_1=X_1$, and, inductively, by
 \be
 \label{1}
 T_{n+1}=\inf\Big\{m>T_n: X_{m}>R_n\Big\},
 \ \ \
 R_{n+1}=X_{T_{n+1}} \ \ \  (n=1,2,\ldots).
 \ee
 It is obvious that \eqref{1} produces an infinite sequence of records
 ($=$ record values) if and only if $F$ has not an atom at its upper end-point
 (if finite).
 In a similar manner, one can define the so called {\it weak} (upper) records,
 $W_n$, by
 $\widetilde{T}_1=1$, $W_1=X_1$, and
 \be
 \label{2}
 \widetilde{T}_{n+1}=\min\Big\{m>\widetilde{T}_{n}:
 X_{m}\geq W_{n}\Big\},
 \ \ \
 W_{n+1}=X_{\widetilde{T}_{n+1}} \ \ \  (n=1,2,\ldots);
 \ee
 clearly, the sequence $W_n$ in \eqref{2} is non-terminating for every d.f.\
 $F$.

 These models have been studied extensively in the literature.
 The interested reader is referred to the books by Ahsanullah (1995),
 Arnold {\it et al.}\
 (1998) and Nevzorov (2001). Moreover,
 several characterization results based on the regressions
 of (weak or ordinary) record values are given in a number of papers,
 including Nagaraja (1977, 1988), Korwar (1984),
 Stepanov (1993), Aliev (1998),
 Dembi\'{n}ska and Wesolowski (2000),
 Lopez-Blazquez and Wesolowski (2001), Raqab (2002),
 Danielak and Dembi\'{n}ska (2007) and Yanev (2012).

 Clearly, the record processes \eqref{1}
 and \eqref{2} coincide with probability (w.p.)\ 1
 whenever $F$ is continuous (i.e., free of atoms). In that case,
 the record process $(R_1,R_2\ldots)$ has the same distribution
 as the sequence
 \be
 \label{3}
 \Big(F^{-1}(U_1),F^{-1}(U_2),\ldots\Big),
 \ee
 where $U_1<U_2<\cdots$ is the record process from the standard uniform d.f.,
 $U(0,1)$, and $F^{-1}(x)=\inf\{x\in\RR:F(x)\geq u\}$, $0<u<1$, is
 the left-continuous inverse d.f.\ of $F$.
 It should be noted, however, that the records, as defined by \eqref{3},
 are neither weak nor ordinary records (when $F$ is arbitrary).
 To illustrate the situation, consider the case where $F$ is symmetric
 Bernoulli, $b(1/2)$, that is,
 $X=0$ or $1$ w.p.\ $1/2$. Then,
 \[
 F^{-1}(u)=\left\{
 \begin{array}{ll}
 0, & 0<u\leq 1/2,
 \\
 1, & 1/2<u<1.
 \end{array}
 \right.
 \]
 The following table provides a realization of the corresponding
 i.i.d.\ and record processes.

 \begin{center}
 {\bf Table 1.}
 {
  \footnotesize
 \
 \\
 \hspace*{-2ex}
 \begin{tabular}{lccccccccc}
 \hline
 \hline
 Random \vspace{-1ex} mechanism \\
 producing i.i.d.\ from $U(0,1)$
 & 0.13 & 0.32 & 0.01 & 0.44 & 0.57 & 0.52 & 0.64 & 0.12 & $\ldots$
 \\
 \hline
 Uniform
 records
 $U_n$
 & 0.13 & 0.32 & $*$ & 0.44 & 0.57 & $*$ & 0.64 & $*$ & $\ldots$
 \\
 \hline
 Records
 $F^{-1}(U_n)$ \mbox{-- see \eqref{3}}
 & 0 & 0 & $*$ & 0 & 1 & $*$ & 1 & $*$ & $\ldots$
 \\
 \hline
 i.i.d.\ observations from $b(1/2)$
 & 0 & 0 & 0 & 0 & 1 & 1 & 1 & 0 & $\ldots$
 \\
 \hline
 Weak records $W_n$ \vspace{-1ex}
 from \\
 the i.i.d.\
 observations \mbox{-- see \eqref{2}}
 & 0 & 0 & 0 & 0 & 1 & 1 & 1 & $*$ & $\ldots$
 \\
 \hline
 Ordinary records $R_n$ \vspace{-1ex}
 from \\
 the i.i.d.\
 observations \mbox{-- see \eqref{1}}
 & 0 & $*$ & $*$ & $*$ & 1 & $*$ & $*$ & $*$ & $\ldots$
 \\
  \hline
 \end{tabular}
 \smallskip
 }
 \end{center}

 \noindent
 Table 1 shows that $W_2=F^{-1}(U_2)=0$ while $R_2=1$. Also,
 $W_4=0$ while $F^{-1}(U_4)=1$ (and $R_4$ is undefined); thus,
 $F^{-1}(U_n)$ is neither $R_n$ nor $W_n$ in general.

 From now on we shall constantly use
 the notation $R_n$ for $F^{-1}(U_n)$, where $\{U_n\}_{n=1}^\infty$
 is the sequence of uniform records -- the effect is not essential
 in applications, where it is customarily assumed that $F$ is absolutely
 continuous. Of course, the three notions of records coincide (w.p.\ 1)
 if and only if
 $F^{-1}(u)$ is strictly increasing in $u\in(0,1)$, and this is equivalent
 to the fact that $\Pr(X=x)=0$ for all $x$.

 The present work is concentrated on questions of the form
 \begin{center}
 {\it Does a given real sequence $\{\rho_n\}_{n=1}^{\infty}$ represents
 an expected
 \\
 record sequence (ERS) of some r.v.\ $X$?}
 \end{center}
 That is, can we find an r.v.\ $X$ such that $\E R_n=\rho_n$
 for all $n$, where the record process $R_n$ is defined by \eqref{3}?
 Moreover, if the answer is in the affirmative,
 is this r.v.\ unique?
 How can we re-construct it
 from its ERS? 

 One of the central results of the paper reads as follows.

 \begin{theo}
 \label{th.1}
 A real sequence $\{\rho_n\}_{n=1}^{\infty}$
 is an expected record sequence
 corresponding to a non-degenerate r.v.\ X
 if and only if
 \be
 \label{6}
 \frac{\rho_{n+2}-\rho_{n+1}}{\rho_2-\rho_1}=\frac{\E T^n}{(n+1)!}, \ \
 n=0,1,\ldots,
 \ee
 for some
 r.v.\ $T$, with $\Pr(T>0)=1$,
 possessing finite moments of any order.
 \end{theo}

 Characterizations of the parent
 distribution through its expected records (under mild additional
 assumptions like continuity and finite moment of order greater than one)
 are present in the bibliography for a long time, the most relevant being those
 given
 by Kirmani and Beg (1984) and Lin
 (1987); see also Lin and Huang (1987).
 However, these authors do not provide an explicit
 connection to the Stieltjes moment problem. In the contrary, the
 corresponding theory for an
 {\it expected maxima sequence}, EMS,
 $\mu_n=\E \max\{X_1,\ldots,X_n\}$,
 is well-understood from Kadane (1971, 1974). Namely, Kadane showed that
 $\{\mu_n\}_{n=1}^{\infty}$ represents an EMS (of a non-degenerate,
 integrable,
 parent population) if and only if there exists a random variable
 $T$, with $\Pr(0<T<1)=1$, such that
 \be
 \label{3b}
 \frac{\mu_{n+2}-\mu_{n+1}}{\mu_2-\mu_1}=\E T^n, \
 \ n=0,1,\ldots \
 .
 \ee
 The representation \eqref{3b} is
 closely connected to the Hausdorff (1921) moment problem,
 and improves on Hoeffding's (1953)
 characterization.
 The above kind of results enable further applications
 in the theory of maxima and order statistics, see, e.g.,
 Hill and Spruill (1994, 2000), Huang (1998), Kolodynski
 (2000). Moreover, the r.v.\
 $T$ in \eqref{3b}, (the distribution of) which is clearly unique,
 admits the
 representation
 $T=F(V)$ where $F$ is the parent d.f.\ and
 $V$ has density $f_V(x)=F(x)(1-F(x))\Big/\int_\RR F(y)(1-F(y))dy$
 -- cf.\ Papadatos (2017). Conversely, the
 parent distribution is characterized from the
 sequence $\{\mu_n\}_{n=1}^{\infty}$,
 and its location-scale family from
 $T$.

 In the case of a record process we would like to
 verify similar results, guaranteing that the theory of maxima
 can be suitably adapted to that of records. However, there are essential
 differences between these two models --
 see, e.g.,
 Resnick (1973, 1987), Nagaraja (1978), Tryfos and Blackmore (1985),
 Embrechts {\it et al.}\ (1997), Papadatos (2012) or
 Barakat {\it et al.}\ (2019); see also
 Appendix \ref{app.a}.
 In this spirit,
 \eqref{6} can be viewed as the natural
 record-analogue of \eqref{3b}.

 The results presented here are based on a suitable mapping
 $\varphi$ on the distribution of a random variable. Using $\varphi$,
 the location-scale family of any suitable $X$ is transformed
 to (the distribution of) a unique positive random variable $T$ with
 finite moments of any order.
 The mapping
 is one to one and onto (hence, invertible),
 and several properties of the expected record sequence of $X$
 are easily extracted from the behavior of $T=\varphi(X)$.
 The basic properties of the mapping $\varphi$ are discussed in
 Section \ref{sec.6}. Using them, we provide a
 complete description of the class of r.v.'s that are characterized
 from their expected record sequence -- see Theorem \ref{th.6.2}.
 Moreover, under mild assumptions,
 an inversion formula for the
 distribution function of the random variable that admits
 a given expected record sequence
 is obtained; see Theorem \ref{th.6.3}. The main
 results are
 presented in Section \ref{sec.6}, and the proofs
 together with some auxiliary lemmas
 are postponed to the appendices.

 Through the rest of the article, $X=Y$ for r.v.'s $X,Y$
 means that $X,Y$ are identically distributed, and inverse d.f.'s are
 always taken to be left-continuous,
 namely, $F^{-1}(u):=\inf\{x:F(x)\geq u\}$, $u\in(0,1)$.

 \section{The mapping $\varphi$ with applications to characterizations}
 \label{sec.6}
 \setcounter{equation}{0}
 For the investigation of the mapping $\varphi$ it is necessary to introduce
 two suitable spaces of r.v.'s.
 \begin{defi}
 \label{def.6.1}
 A function $H:(0,\infty)\to\RR$ belongs to
 ${\cal H}^*$ if it is non-constant, non-decreasing, left continuous, and
 satisfies
 \[
 \int_{0}^{\infty} y^m e^{-y} \big|H(y)\big| dy<\infty, \ \ m=0,1,\ldots \ \ .
 \]
 Furthermore, ${\cal H}_0:=\Big\{H\in {\cal H^*}:\int_{0}^\infty e^{-y} H(y) dy=0
 \ \mbox{and} \
 \int_{0}^\infty y e^{-y} H(y) dy=1\Big\}$. By definition, a random variable
 $X$ belongs to ${\cal H}^*$ if its inverse d.f.\ $G$ can be written as
 $G(u)=H(-\log(1-u))$, $0<u<1$, for some $H\in{\cal H}^*$. Similarly,
 $X_0\in{\cal H}_0$  if its inverse d.f.\ $G_0$ can be written as
 $G_0(u)=H_0(-\log(1-u))$, $0<u<1$, for some $H_0\in{\cal H}_0$.
 Here, identically distributed r.v's are considered as equal. Finally,
 ${\cal H}:={\cal H}^*\cup\{$the constant functions$\}$.
 \end{defi}
 Notice that $X\in{\cal H}^*$ if and only if $X$ is non-degenerate
 and the corresponding record process in \eqref{3} satisfies
 $\E|R_n|<\infty$ for all $n$ --
 see Proposition \ref{prop.1}.
 It is worth pointed out that every $X\in{\cal H}$ admits the
 equivalent representation $X=H({E})$, where the function
 $H$ belongs to ${\cal H}$ and ${E}$ is a standard exponential r.v.
 This says that a left-continuous, non-decreasing function
 $H$ belongs to ${\cal H}$ if and only if $\E|H(S_m)|<\infty$
 for all $m\geq 1$, where $S_m$ follows the Erlang distribution
 with parameters $m$ and $1$, i.e., $S_m$ is the sum of $m$ i.i.d.\ standard
 exponential r.v.'s. The subspace ${\cal H}_0$ consists
 of those $H\in{\cal H}$ for which $\E H(S_1)=0$, $\E H(S_2)=1$.
 \begin{defi}
 \label{def.6.2}
 The space ${\cal T}$ consists of all r.v.'s $T$,
 with $\Pr(T>0)=1$, possessing finite moments of any order, where identically
 distributed r.v.'s are considered as equal.
 We customarily write $F_T\in{\cal T}$
 in order
 to denote $T\in{\cal T}$, where
 $F_T$ is the d.f.\ of $T$.
 \end{defi}

 We are now ready to define the mapping $\varphi$ and its inverse
 $\varphi'$. What we shall prove in the sequel is that, essentially,
 the spaces ${\cal H}_0$ and ${\cal T}$ are identified
 through the restriction of $\varphi$ on ${\cal H}_0$.
 \begin{defi}
 \label{def.6.3}
 Set $L(u)=-\log(1-u)$, $0<u<1$.
 For $X\in{\cal H}^*$, $T=\varphi(X)\in{\cal T}$
 is  defined  to be the r.v.\
 $T=L(F(V))$, where $F$ is the d.f.\ of $X$
 and the r.v.\ $V$ has density (with respect to Lebesgue measure)
 given by $f_V(x)=(1-F(x))L(F(x))\Big/\int_{\RR} (1-F(y))L(F(y)) dy$, $x\in\RR$,
 and where $(1-F(x))L(F(x))=0$ if $F(x)=1$ or $0$.
 \end{defi}
 The mapping $\varphi$ is well-defined because $X\in{\cal H}^*$
 so that $f_V$ is integrable,  strictly positive in
 the non-empty interval $\{x: 0<F(x)<1\}$, and zero otherwise.

 \begin{defi}
 \label{def.6.4} For any $T\in{\cal T}$ with d.f.\
 $F_T$ we define $X_0=\varphi'(T)\in{\cal H}_0$
 to be the r.v.\ with inverse d.f.\ $G_0$, for which
 the function $H_0(y)=G_0(1-e^{-y})$, $y>0$, is given by
 the formula
 \be
 \label{6.1}
 H_0(y):=\frac{e^y}{y} F_T(y-)-\int_{1}^{y} \frac{x-1}{x^2} e^x F_T(x) dx
 -c_T,
 \ \ \ y>0.
 \ee
 In this formula, $F_T(y-)=\Pr(T<y)$, $\int_{1}^{y}dx:=\int_{y}^1-dx$ if
 $0<y<1$, and $c_T$ is the unique constant (depending only on $T$)
 for which $\int_0^\infty e^{-y}H_0(y)dy=0$.
 \end{defi}

 We shall prove in Lemma \ref{lem.cT} that
 \be
 \label{cT}
 c_T=\E\Big[\frac{1}{T}I(T> 1)\Big]
 +\E\Big[\frac{1+eT-e^T}{T}I(T\leq 1)\Big],
 \ee
 where $I$ denotes an indicator function.

 \begin{prop}
 \label{prop.phi}
 Both transformations $\varphi$ and $\varphi'$
 are well-defined, with domains ${\cal H}^*$ and ${\cal T}$, and
 ranges ${\cal T}$ and ${\cal H}_0$, respectively. Moreover,
 if $X_1\in{\cal H}^*$ and $X_2=c+\lambda X_1$, where
 $c\in\RR$ and $\lambda>0$,
 then $X_2\in{\cal H}^*$ and $\varphi(X_1)=\varphi(X_2)$.
 \end{prop}

 \begin{theo}
 \label{th.6.1}
 The transformation $\varphi:{\cal H}_0\to{\cal T}$
 of Definition {\rm \ref{def.6.3}}
 is one to one and onto, with inverse
 $\varphi^{-1}:{\cal T}\to{\cal H}_0$,
 where $\varphi^{-1}=\varphi'$ is given by
 Definition {\rm \ref{def.6.4}}.
 \end{theo}
 Theorem \ref{th.1} holds true, since it is an immediate
 corollary of the following result.
 \begin{theo}
 \label{th1}

 \noindent
 {\rm (i)} Given $X\in {\cal H}^*$ with expected record sequence
 $\{\rho_n\}_{n\geq 1}$,
 the r.v.\ $T=\varphi(X)\in{\cal T}$ satisfies \eqref{6},
 where the mapping $\varphi$ is given by Definition {\rm \ref{def.6.3}}.

 \noindent
 {\rm (ii)} Given $T\in {\cal T}$,
 the r.v.\ $X_0=\varphi^{-1}(T)\in{\cal H}_0$ has expected record sequence
 $\{\rho_n\}_{n\geq 1}$ that satisfies \eqref{6}
 with $\rho_1=0$, $\rho_2=1$,
 where the mapping $\varphi^{-1}=\varphi'$ is as in
 Definition {\rm \ref{def.6.4}}.
 \end{theo}

 \begin{rem}
 \label{rem.0}
 Given $T\in{\cal T}$, $\mu\in\mathds{R}$, and $\lambda>0$,
 we can always construct an r.v.\ $X\in{\cal H}^*$
 with ERS $\{\rho_n\}_{n\geq1}$ satisfying \eqref{6},
 with $\rho_1=\mu$, $\rho_2-\rho_1=\lambda$,
 namely,
 $X:=\mu+\lambda \varphi^{-1}(T)$.
 \end{rem}
 In the particular case where $T\in{\cal T}$ admits a density,
 the inversion formula \eqref{6.1} simplifies considerably, after an
 obvious application of Tonelli's theorem.
 \begin{cor}
 \label{cor.th.2}
 If the r.v.\ $T\in{\cal T}$
 has a density $f_T$
 then
 the function $H_0$ in \eqref{6.1} is given by
 \be
 \label{5.1}
 H_0(y)=\int_{1}^{y}  \frac{e^x}{x} f_T(x) dx
 +\int_{0}^1 \frac{e^t-1}{t} f_T(t)  dt-
 \int_{1}^\infty \frac{1}{t} f_T(t)dt,
 \ \ \ \ \ y>0.
 \ee
 Moreover,
 the r.v.\  $X_0=\varphi^{-1}(T)$
 has a continuous d.f.\ if $f_T(y)>0$
 for almost all $y>0$.
 \end{cor}

 \begin{rem}
 \label{rem.mono}
 The formula \eqref{5.1} is unable to describe several continuous
 r.v.'s in ${\cal H}_0$, for which, however, the ordinary record process
 $\{R_n\}_{n\geq 1}$ is well-defined. This is so because
 any r.v.\ $T\in{\cal T}$ with dense support in $(0,\infty)$ will produce
 a continuous r.v.\ $X_0=\varphi^{-1}(T)\in{\cal H}_0$. This observation
 is a consequence of \eqref{mono}, which implies that, for such an r.v.\ $T$,
 $H_0$ is strictly increasing, and hence, its d.f.\ is continuous. It
 is obvious that we can find discrete r.v.'s $T$ with dense support
 and finite moment generating function at a neighborhood of zero.
 As a concrete example,
 set $T=T_1+T_2$, where  $T_1$ follows a Poisson d.f.\ with mean $1$,
 $\Pr(T_2=r_n)=2^{-n}$ ($n=1,2,\ldots$), with $\{r_1,r_2,\ldots\}$
 being an enumeration of the rationals of the interval $(0,1]$,
 and assume that $T_1,T_2$ are independent. Set also
 $X=\varphi^{-1}(T)$.
 The following theorem
 shows that
 this particular continuous r.v.\ $X$
 is,
 indeed, characterized from its ERS.
 \end{rem}

 With the aim of mapping $\varphi$,
 a complete characterization result
 based on the expected record sequence becomes possible, as follows.

 \begin{theo}
 \label{th.6.2}
 A random variable $X\in{\cal H}^*$ is characterized from
 its expected record sequence if and only if the random variable
 $T=\varphi(X)\in{\cal T}$ is characterized from its moments,
 where the mapping $\varphi$
 is given by Definition {\rm \ref{def.6.3}}.
 \end{theo}

 Suppose that for a given (non-degenerate) r.v.\ $X$,
 $\E X^{-}<\infty$ and $\E (X^+)^p<\infty$ for some $p>1$.
 According to Theorem \ref{th.6.3}, below, the transformation
 $T=\varphi(X)$ of any such r.v.\
 has finite moment generating function at a neighborhood of zero; hence
 $T$ it characterized from its moments, and we obtain
 the following result.

 \begin{cor}
 \label{cor.1} (Kirmani and Beg, 1984).
 Every random variable $X$ with finite absolute moment of order $p>1$
 is characterized
 from its
 expected record sequence.
 \end{cor}
 However, we emphasize that the Kirmani-Beg characterization
 do not extends to ${\cal H}^*$:
 \begin{exam}
 \label{exam.1}
 There exist different r.v.'s in ${\cal H}_0$ with identical
 expected record sequence.
 A concrete example leading to absolutely continuous
 r.v.'s can be constructed
 by means of the classical example due to Stieltjes,
 as follows. Let $T$ be the lognormal r.v.\ with
 density $f_T(t)=e^{-(\log t)^2/2}/(t\sqrt{2\pi})$, $t>0$,
 and moments $\E T^n=e^{n^2/2}$. Each density in the set
 $\Big\{f_\lambda(t):=(1+\lambda \sin(\pi\log t))f_T(t),
 \ -1\leq \lambda\leq 1\Big\}$ admits the same moments as
 $T$ -- see Stoyanov (2013) or Stoyanov and Tolmatz (2005).
 Assume that $T_{\lambda}$ has density $f_\lambda$, and consider
 the r.v.\ $X_\lambda=\varphi^{-1}(T_\lambda)$,
 with distribution inverse
 given by
 \[
  G_{\lambda}(u)
  :=\int_{1}^{-\log(1-u)}
  \frac{e^y}{y} f_\lambda(y) dy
  +\int_{0}^1 \frac{e^t-1}{t} f_\lambda(t) dt
  -\int_{1}^\infty \frac{1}{t} f_\lambda(t) dt,
 \]
 $0<u<1$.
 Using an obvious notation, it is clear from
 Theorem {\rm \ref{th1}(ii)} and
 Corollary {\rm \ref{cor.th.2}}
 that
 $\E R_1(X_\lambda)=0$, $\E R_2(X_\lambda)=1$,
 and the sequence $\rho_n^{(\lambda)}:=\E R_n(X_\lambda)$ satisfies
 \eqref{6} with $T_\lambda$ in place of $T$.
 Thus, each $X_{\lambda}$, $-1\leq \lambda\leq 1$,
 has the same expected record sequence,
 namely,
 \[
 \rho_n^{(\lambda)}
 =\int_{0}^1 \frac{[-\log(1-u)]^{n-1}}{(n-1)!} G_{\lambda}(u) du
 =\sum_{k=0}^{n-2} \frac{e^{k^2/2}}{(k+1)!},
 \ \ n=1,2,3,\ldots, \ \  -1\leq \lambda\leq 1,
 \]
 where an empty sum should be treated as zero.
 Differentiating $G_\lambda$, it follows
 that $G_{\lambda_1}\neq G_{\lambda_2}$ for
 $\lambda_1\neq \lambda_2$. Therefore, the
 function $g:=G_{\lambda_1}-G_{\lambda_2}$
 belongs to
 \[
 {\cal H}(0,1):=\left\{ g:(0,1)\to\RR:
 \int_0^1 |g(u)| [-\log(1-u)]^k  du<\infty \ \ \mbox{ for }
 \ \ k=0,1,\ldots \right\},
 \]
 it is non-zero
 (when $\lambda_1\neq \lambda_2$)
 in a set of positive measure, and satisfies
 \be
 \label{inc}
 \int_0^1 g(u) [-\log(1-u)]^k  du=0, \ \ \ k=0,1,\ldots \ .
 \ee
 It is easily checked that every $X_\lambda$
 admits a density.

 In fact, the Kirmani-Beg characterization holds true because the system
 of functions
 $
 {\cal L}:=\Big\{ [-\log(1-u)]^k, \ \ k=0,1,\ldots\Big\}
 $
 is complete in $\cup_{\delta>0}L^{1+\delta}(0,1)$, see Lemma 3 in Lin
 {\rm (1987)},
 while \eqref{inc} implies that
 ${\cal L}$
 is not complete in the larger space ${\cal H}(0,1)$.
 \end{exam}


 Our final result is applicable to most practical
 situations regarding characterizations (and inversions)
 in terms of the expected record sequence.
 \begin{theo}
 \label{th.6.3}
 Let $X\in{\cal H}^*$ with ERS $\{\rho_n\}_{n\geq 1}$, set
 $w_n:=\E |R_n|$, $T=\varphi(X)$,
 and define the following generating functions:
 \begin{eqnarray*}
 G_{\rho}(t)&:=&\mbox{$\sum_{n\geq 1} \rho_n t^{n-1}$}
 \hspace{4ex}
 \mbox{\rm [the generating function of $\rho_n$'s]}
 \\
 G_{w}(t)&:=&\mbox{$\sum_{n\geq 1} w_n t^{n-1}$}
 \hspace{3.75ex}
 \mbox{\rm [the generating function of $w_n$'s]}
 \\
 M_T(a)&:=&\E e^{a T}
 \hspace{9.5ex}
 \mbox{\rm [the moment generating function of $T$].}
 \end{eqnarray*}
 Then,
 the following statements are equivalent.

 \noindent
 {\rm (i)} \ \ $\E(X^+)^p<\infty$ for some $p>1$.
 \\
 {\rm (ii)} \ $G_{\rho}(t)$ is finite for $t$ in a neighborhood of zero.
 \\
 {\rm (iii)} \ $G_{w}(t)<\infty$ for some $t>0$.
 \\
 {\rm (vi)} \ $M_T(a)<\infty$ for some $a>0$.

 \noindent
 If {\rm (i)--(iv)} hold, then we can find $\epsilon_0\in(0,1)$
 such that
 \be
 \label{6.8}
 G_{\rho}(t)=\frac{1}{1-t}
 \left\{\rho_1+(\rho_2-\rho_1) \int_{0}^t M_T(s) ds\right\},  \ \
 |t|<\epsilon_0.
 \ee
 Consequently,
 \be
 \label{6.9}
 M_T(t)=\frac{1}{\rho_2-\rho_1}\ \frac{d}{dt}\Big((1-t)G_{\rho}(t)\Big),
 \ \
 |t|<\epsilon_0.
 \ee
 Therefore, under assumption {\rm (i)},
 $X$ is characterized from the generating function
 $G_{\rho}$ of its ERS through
 $X=\rho_1+(\rho_2-\rho_1)\varphi^{-1}(T)$,
 where $T$
 has moment generating function $M_T$, given by \eqref{6.9}, and
 $\varphi^{-1}=\varphi'$ as in Definition {\rm \ref{def.6.4}}.
 \end{theo}

 It is well-known
 that any r.v.\ $T$ is uniquely determined from its
 moments, if it admits a finite moment generating function at
 a neighborhood of
 zero. On the other hand, it is also known
 that we can find several r.v.'s $T\in{\cal T}$
 that are characterized from their moment sequence, although
 $\E e^{aT}=\infty$ for all $a>0$. A large family of such r.v.'s
 is the so called {\it Hardy class} --
 see Stoyanov (2013) and Lin and Stoyanov (2016)
 for more details.
 Clearly, the corresponding r.v.'s
 $\varphi^{-1}(T)$ are not treated by Kirmani-Beg's (1984)
 characterization,
 showing that the proposed method, based on the transformation $\varphi$,
 is quite efficient.

\section*{Acknowledgements}
I would like to thank A.\ Giannopoulos for helpful discussions that
led to the results of Theorem
\ref{th.6.3}.

 \appendix

 \section{Existence of expectations of records}
 \setcounter{equation}{0}
 \label{app.a}

 It is well-known (see, e.g., Arnold {\it et al.}, 1998)
 that $U_n$ in \eqref{3} has density
 \be
 \label{4}
 f_{U_n}(u)=\frac{L(u)^{n-1}I(0<u<1)}{(n-1)!},  \ n=1,2,\ldots,
 \ \mbox{where} \  L(u):=-\log(1-u), \ 0<u<1,
 \ee
 and $I$ denotes the indicator function.
 We may use \eqref{4} to calculate the d.f.\ $F_n$ of $R_n$ as follows:
 \[
 F_n(x)=\Pr\Big(F^{-1}(U_n)\leq x\Big)
 =\Pr\Big(U_n\leq F(x)\Big)
 =\frac{1}{(n-1)!}\int_0^{F(x)} L(u)^{n-1}
 du.
 \]
 Substituting $L(u)=y$ in the integral we see that
 $F_n(x)=\Pr\Big(E_1+\cdots+E_n \leq L(F(x))\Big)$,
 where $E_1,\ldots,E_n$ are i.i.d.\ from the standard exponential,
 $\mbox{Exp}(1)$.
 From the well-known relationship regarding waiting times for the
 standard (with intensity one) Poisson
 process, $\{Y_{t}, t\geq 0\}$,
 we have
 \[
 \Pr\Big(E_1+\cdots+E_n \leq t \Big)=1-\Pr(Y_{t}\leq n-1)
 =1-e^{-t}\sum_{k=0}^{n-1} \frac{t^k}{k!}, \ \ \ \ t\geq 0.
 \]
 Therefore, with $t=L(F(x))$, we obtain
 (cf.\ Nagaraja, 1978)
 \be
 \label{5}
 F_n(x)
 =
 1-(1-F(x))\sum_{k=0}^{n-1} \frac{L(F(x))^k}{k!}, \ \ \ x\in\RR
 \ \ \ (n=1,2,\ldots).
 \ee
 In the above sum, the term $L(F(x))^0$ should be treated as $1$ for all $x$;
 moreover, the product $(1-F(x))L(F(x))^k$ should be treated as $0$
 whenever $k\geq 1$
 and $F(x)=1$.
 Hence, \eqref{5} yields $F_1(x)=F(x)$ and, e.g.,
 \[
 F_2(x)=\left\{
 \begin{array}{ll}
 1-\Big(1-F(x)\Big)\Big(1+L(F(x))\Big), & \mbox{if } \ F(x)<1,
 \\
 1, & \mbox{if } \ F(x)=1.
 \end{array}
 \right.
 \]

 Since our problem concerns the expectations $\E R_n$ for all $n$,
 we have to define
 an appropriate space to work with; that is, to
 guarantee that these expectations are, all, finite.
 The natural space is given by Definition \ref{def.6.1},
 since the next proposition holds true.
 %
 \begin{prop}
 \label{prop.1}
 The following statements are equivalent:

 \noindent
 {\rm (i)} $X\in{\cal H}$, i.e., $H\in{\cal H}$,
 where $H(y):=F^{-1}(1-e^{-y})$, $y>0$, and $F$ is the d.f.\ of $X$.

 \noindent
 {\rm (ii)} $\E |R_n|<\infty$ for all $n=1,2,\ldots$ \ .

 \noindent
 {\rm  (iii)} $\E X^{-}<\infty $ and $\E X (\log^{+}X)^m<\infty$
 for all $m>0$, where $\log^{+} x=\log x$ if $x\geq1$ and
 $=0$ otherwise.

 \noindent
 {\rm (iv)} $\int_{0}^1 L(u)^{m}|F^{-1}(u)|du<\infty$, \ \ $m=0,1,\ldots$ \ .

 \noindent
 {\rm (v)} $\int_{0}^\infty y^{m}e^{-y}|H(y)|dy<\infty$, \ \ $m=0,1,\ldots$ \ .

 \noindent
 If {\rm (i)--(v)} are satisfied, then the sequence $\rho_n=\E R_n$
 is given by
 \be
 \label{E}
 \rho_n=\int_{-\infty}^\infty (I(x>0)-F_n(x)) dx
 = \int_{0}^1 \frac{L(u)^{n-1}}{(n-1)!} F^{-1}(u) du
 =\int_{0}^\infty \frac{y^{n-1}}{(n-1)!}e^{-y}H(y)dy,
 \ee
 $n=1,2,\ldots$,
 with $F_n$ given by \eqref{5}
 and $L$ by \eqref{4}.
 \end{prop}

 \begin{prop}
 \label{prop.2}
 For $\alpha\geq 1$ set
 ${L}^\alpha =\left\{X:\E |X|^\alpha<\infty\right\}$,
 where identically distributed r.v.'s are considered as equal.
 Then,
 $\cup_{\delta>0}{L}^{1+\delta}
 \varsubsetneqq
 {\cal H}
  \varsubsetneqq
   {L}^{1}.
 $
 \end{prop}

 These results are due to Nagaraja (1978) in the particular case where
 $X$ has a density and/or is non-negative, but his proofs continue to
 hold in our case too.

 \section{The transformation}
 \setcounter{equation}{0}
 \label{app.b}

 \begin{lem}
 \label{lem.6.2} If $X_1,X_2\in{\cal H}_0$ and $\varphi(X_1)=\varphi(X_2)$
 then $X_1=X_2$.
 \end{lem}
 \begin{proof}
 Let $F_i$ (resp., $F_i^{-1}$) be the d.f.\ (resp., the inverse d.f.)
 of $X_i$, and $V_i$ the corresponding r.v.\
 with density $f_i=(1-F_i)L(F_i)$ ($i=1,2$). It is easy to verify
 that the events $\{L(F_i(V_i))<t\}$ and $\big\{V_i<F_i^{-1}(1-e^{-t})\big\}$
 are equivalent for all $t>0$. Setting $T_i=\varphi(X_i)$ and $u=1-e^{-t}$,
 the assumption
 $T_1=T_2$ is equivalent to $\Pr(T_1<t)=\Pr(T_2<t)$ for all $t>0$, i.e.,
 \be
 \label{6.2}
 \int_{-\infty}^{F_1^{-1}(u)}(1-F_1(x)) L(F_1(x)) dx=
 \int_{-\infty}^{F_2^{-1}(u)}(1-F_2(x)) L(F_2(x)) dx, \ \ 0<u<1.
 \ee
 For every r.v.\ $X$ with
 d.f.\ $F$ and inverse d.f.\ $F^{-1}$, the following identity is valid
 (see, e.g., Lemma 4.1 in Papadatos, 2001):
 \be
 \label{6.3}
 \int_{-\infty}^{F^{-1}(u)} F(x)^m dx
 =m \int_{0}^u w^{m-1}\Big(F^{-1}(u)-F^{-1}(w)\Big) dw,
 \ \ 0<u<1.
 \ee
 Using \eqref{6.3} and assuming $\E X^-<\infty$, i.e.,
 $F^{-1}\in L^1(0,1-\delta)$
 for any $\delta\in (0,1)$,
 we obtain
 \begin{eqnarray}
 \nonumber
 &&
 \hspace{-7ex}
  \int_{-\infty}^{F^{-1}(u)} (1-F(x))L(F(x)) dx
  =  \sum_{m=1}^\infty \frac{1}{m}
  \int_{-\infty}^{F^{-1}(u)} (1-F(x))F(x)^m dx
  \\
  \nonumber
  &&
 =  \sum_{m=1}^\infty \int_0^{u} w^{m-1} \Big(F^{-1}(u)-F^{-1}(w)\Big)
 \Big(1-(1+1/m)w\Big) dw
 \\
 \label{6.4}
 &&
 =  \int_0^{u} \Big(F^{-1}(u)-F^{-1}(w)\Big)
 \Big(1-L(w)\Big) dw.
 \end{eqnarray}
 Define $s(w):=1-L(w)$ and
 $S(w):=\int_{0}^w s(t) dt$, so that $S'(w)=s(w)$, $0<w<1$.
 In view of \eqref{6.4}, \eqref{6.2} reads as
 \[
  h(u)=\frac{1}{S(u)}\int_{0}^{u}s(w) h(w) dw, \ \ 0<u<1,
  \]
  where $h:=F_1^{-1}-F_2^{-1}$. This relation shows that
  $h$ is absolutely continuous
  in every compact interval $[x,y]\subseteq(0,1)$, and
  \[
  h'(u)=\frac{s(u)}{S(u)}
  \left(h(u)-\frac{1}{S(u)}\int_{0}^{u}s(w) h(w) dw\right)=0, \ \
  \mbox{for almost all } \ u\in(0,1).
  \]
  Therefore, $h=c$, constant. Finally,
  from the assumption $X_1,X_2\in {\cal H}_0$, we must have
  $c=\int_{0}^1 h(u) du=0$; hence, $F_1^{-1}=F_2^{-1}$, i.e.,
  $X_1=X_2$.
 \end{proof}

 \begin{rem}
 \label{rem.6.1}
 The equation \eqref{6.4} provides an explicit expression
 for the d.f.\ of $T=\varphi(X)$ when $X\in{\cal H}_0$, namely,
 \[
  \Pr(T<t)=\int_{0}^{1-e^{-t}} \big(F^{-1}(1-e^{-t})-F^{-1}(w)\big)
  \big(1-L(w)\big) dw, \ \ t>0.
 \]
 \end{rem}

 \begin{lem}
 \label{lem.6.3} If
 $T_1,T_2\in{\cal T}$ and $\varphi'(T_1)=\varphi'(T_2)$
 then $T_1=T_2$.
 \end{lem}
 \begin{proof}
 Write $H_i(y)$ for $G_i(1-e^{-y})$, $y>0$, where
 $G_i$ is the inverse d.f.\ of $X_i=\varphi'(T_i)\in {\cal H}_0$
 ($i=1,2$). From the proof in Appendix \ref{app.d}
 we know that $H_i\in{\cal H}_0$. 
 Note that
 $H_i$ is the function $H_0$ given by \eqref{6.1},
 on substituting $T=T_i$ ($i=1,2$). By assumption, $H_1=H_2$.
 Thus,
 \[
  F_{T_1}(y-)-F_{T_2}(y-)= y e^{-y}
  \int_{1}^{y} \frac{x-1}{x^2} e^x \big(F_{T_1}(x)-F_{T_2}(x)\Big)dx +c y e^{-y},
  \ \
  y>0,
 \]
 where $c=c_{T_1}-c_{T_2}$ (see \eqref{cT}) and $F_{T_i}$ is the d.f.\ of $T_i$.
 Setting $h=F_{T_1}-F_{T_2}$,
 the above relation implies that $h$ is absolutely continuous
 in every compact interval $[y_1,y_2]\subseteq(0,\infty)$, so that
 $h(y-)=h(y)$ for all $y>0$, yielding
 \[
 w(y)- \int_{1}^{y} \frac{x-1}{x} w(x) dx =c,
  \ \ \
  y>0,
 \]
 where $w(y):=e^y h(y)/y$. Hence, $0=w'(y)-(y-1)w(y)/y=e^y h'(y)/y$ for almost
 all $y\in(0,\infty)$. Thus, $h'(y)=0$ and, therefore,
 $h=c_1$, constant. Finally, taking limits as $y\to\infty$,
 we conclude that $c_1=0$
 and $F_{T_1}=F_{T_2}$.
 \end{proof}

  \begin{proof}[Proof of Theorem \ref{th.6.1}]
 In view of Lemmas \ref{lem.6.2}, \ref{lem.6.3}, and
 the sufficiency proof of Theorem \ref{th.1} -- see Appendix \ref{app.d} --
 it remains to verify that $\varphi'(\varphi(X))=X$ for every $X\in{\cal H}_0$.
 If this is proved, then $\varphi(\varphi'(T))=T$ for each $T\in {\cal T}$,
 and thus, $\varphi'=\varphi^{-1}$.
 To see that $\varphi'(\varphi(X))=X$ implies
 $\varphi(\varphi'(T))=T$, set $T_1=\varphi(\varphi'(T))\in {\cal T}$.
 Then, $\varphi'(T_1)=\varphi'\Big(\varphi\big(\varphi'(T)\big)\Big)=
 \varphi'(T)$, because $\varphi'(T)\in{\cal H}_0$.
 Thus, $T_1=T$ from the one to one property of $\varphi'$
 -- Lemma \ref{lem.6.3}. Pick now $X\in {\cal H}_0$, and set
 $H(y)=G(1-e^{-y})$, $y>0$, where $G$ is the distribution inverse of $X$;
 set also $T=\varphi(X)$. From Remark \ref{rem.6.1} we have
 \[
 F_T(y-):=\Pr(T<y)=y e^{-y}H(y)-\int_0^{y} (1-w) e^{-w} H(w) dw, \ \  y>0.
 \]
 Let $X_0=\varphi'(T)$, assume that $G_0$ is the distribution inverse of
 $X_0$, and set $H_0(y)=G_0(1-e^{-y})$, $y>0$. Applying \eqref{6.1}
 to $F_T$ we find
 \begin{eqnarray*}
 H_0(y)&=&
 H(y)-\frac{e^y}{y}\int_0^{y} (1-w) e^{-w} H(w) dw
 -\int_{1}^{y} \frac{x-1}{x} H(x) dx
 \\
 &&
 +\int_{1}^{y} \frac{x-1}{x^2} e^x\int_0^{x} (1-w) e^{-w} H(w) dw dx-c_{T},
 \ \ y>0.
 \end{eqnarray*}
 The double integral can be rewritten as
 \begin{eqnarray*}
 &&
 \hspace*{-3.5ex}
 \int_{1}^{y} \frac{x-1}{x^2} e^x\int_0^{1} (1-w) e^{-w} H(w) dw dx+
 \int_{1}^{y} \frac{x-1}{x^2} e^x\int_1^{x} (1-w) e^{-w} H(w) dw dx
 \\
 &&
 =\left(\frac{e^y}{y}-e\right) \int_0^{1} (1-w) e^{-w} H(w) dw
 +
 \int_{1}^{y} (1-w) e^{-w} H(w)\int_w^{y} \frac{x-1}{x^2} e^x dx dw
 \\
 &&
 =\left(\frac{e^y}{y}-e\right) \int_0^{1} (1-w) e^{-w} H(w) dw
 +
 \int_{1}^{y} (1-w) e^{-w} H(w)\left(\frac{e^y}{y}-\frac{e^w}{w}\right) dw
 \\
 &&
 =\frac{e^{y}}{y}\int_{0}^y (1-w) e^{-w} H(w) dw
 +\int_{1}^y \frac{w-1}{w} H(w) dw-c_1,
 \end{eqnarray*}
 where $c_1=e \int_0^{1} (1-w) e^{-w} H(w) dw$, and the change in the order
 of integration
 is justified from Tonelli-Fubini, since $w\to (1-w)e^{-w}H(w)\in L^1(0,\infty)$; see
 Appendix \ref{app.d}. Thus,
 \begin{eqnarray*}
 H_0(y)-H(y)&=&
 -\frac{e^y}{y}\int_0^{y} (1-w) e^{-w} H(w) dw
 -\int_{1}^{y} \frac{x-1}{x} H(x) dx
 \\
 &&
 +\frac{e^{y}}{y}\int_{0}^y (1-w) e^{-w} H(w) dw+\int_{1}^y \frac{w-1}{w} H(w) dw
 +c_2=c_2,
 \end{eqnarray*}
 where $c_2=-c_{T}-c_1$. Since $c_2=\int_{0}^\infty e^{-y}(H_0(y)-H(y))dy=0$
 (because $X, X_0\in {\cal H}_0$), the theorem is proved.
 \end{proof}

 \section{Proofs of characterizations}
 \label{app.c}
 \noindent
 \begin{proof}[Proof of Theorem \ref{th1}{\rm (i)} and of the first part
 of Proposition \ref{prop.phi}]
 Suppose that $\rho_n=\E R_n$ for all $n$ and
 some $X\in{\cal H}^*$ which
 has d.f. $F$. Then,
 \[
 \rho_n=\int_{-\infty}^{\infty} (I(x>0)-F_n(x)) dx\in\RR,
 \]
 and from \eqref{5} we see that
 \be
 \label{7}
 \rho_{n+1}-\rho_{n}=\int_{-\infty}^{\infty} (F_n(x)-F_{n+1}(x)) dx
 =\frac{1}{n!}\int_{\alpha}^{\omega} (1-F(x)) L(F(x))^n dx, \ \ n=1,2,\ldots,
 \ee
 where $\alpha=\inf\{x:F(x)>0\}$, $\omega=\sup\{x:F(x)<1\}$. Note that
 $F_n(x)-F_{n+1}(x)=0$ for $x\not\in[\alpha,\omega)$. Since
 $\alpha<\omega$ (because $F$ is non-degenerate), the above relation
 shows that
 \[
 \rho_2-\rho_1=\int_{\alpha}^{\omega} (1-F(x)) L(F(x)) dx>0,
 \]
 because $(1-F(x)) L(F(x))>0$ for $x\in(\alpha,\omega)$.
 It follows that the function
 \[
 f_V(x):=\left\{
 \begin{array}{ll}
 (1-F(x))L(F(x))/(\rho_2-\rho_1), & \alpha<x<\omega,
 \\
 0, & \mbox{otherwise},
 \end{array}
 \right.
 \]
 defines a
 Lebesgue density of an absolutely continuous r.v.\ $V$ with support
 $(\alpha,\omega)$. Setting $T:=-\log(1-F(V))=L(F(V))$ we see that
 $0<T<\infty$ w.p.\ $1$ (because $\alpha<V<\omega$
 so that $0<F(V)<1$ w.p.\ $1$).
 Thus, we can rewrite
 \eqref{7}
 as
 \[
 \frac{\rho_{n+2}-\rho_{n+1}}{\rho_2-\rho_1}
 =\frac{1}{(n+1)!}\int_{\mathds R} f_V(x) L(F(x))^n dx
 =\frac{\E L(F(V))^n}{(n+1)!}, \ \ n=0,1,\ldots,
 \]
 and \eqref{6} is proved with $T=L(F(V))=\varphi(X)$; this also verifies
 the first counterpart of Proposition \ref{prop.phi}, i.e., that
 $\varphi$ has domain ${\cal H}^*$ and takes values into ${\cal T}$.
 The fact that $\varphi(X_1)=\varphi(\lambda X_1+c)$ (when
 $X_1\in{\cal H}^*$, $c\in\RR$, $\lambda>0$) is trivial.
 \end{proof}


 \begin{proof}[Proof of Theorem \ref{th.6.2}]
 Assume first that $X\in {\cal H}^*$ is characterized from
 $\{\rho_n\}_{n=1}^\infty$, where $\rho_n=\E R_n$, and
 set $X_0=(X-\rho_1)/(\rho_2-\rho_1)\in{\cal H}_0$.
 If $T=\varphi(X)=\varphi(X_0)\in{\cal T}$ is not characterized
 from its moments, then we can find an r.v.\ $T_1\in{\cal T}$, $T_1\neq T$,
 such that $\E T_1^n=\E T^n$ for all $n\in\{0,1,\ldots\}$. Then,
 from Theorem \ref{th1}(ii),
 the r.v.\
 $X_1:=\varphi^{-1}(T_1)\in {\cal H}_0$ possesses the same expected record
 sequence as $X_0$, and, thus, the r.v.\ $X':=\rho_1+(\rho_2-\rho_1)X_1$
 has the same ERS as $X$. Since $\varphi^{-1}$ is one to one and
 $T_1\neq T$, it follows that $X_1\neq X_0$ and, consequently,
 $X'\neq X$, which contradicts the assumption that $X$ is characterized from
 its ERS.

 Next, assume that $T=\varphi(X)\in {\cal T}$ is characterized from its moments.
 Suppose, in contrary, that $X\in{\cal H}^*$ is not characterized from its ERS.
 Then, we can find
 an r.v.\ $X'\in{\cal H}^*$, $X'\neq X$, with the same ERS as $X$. Obviously,
 if $\{\rho_n\}_{n=1}^\infty$ is the common ERS, both r.v.'s
 $X_0:=(X-\rho_1)/(\rho_2-\rho_1)$ and $X'_0:=(X'-\rho_1)/(\rho_2-\rho_1)$
 belong to ${\cal H}_0$ and posses the same expected record sequence, while
 $X'_0\neq X'$. From
 Theorem \ref{th1}(i), see Appendix \ref{app.d}, it follows that the r.v.'s $T$
 and
 $T':=\varphi(X'_0)=\varphi(X')\in{\cal T}$ posses identical moments. However, since
 $X_0'\neq X_0$ and $X_0,X_0'\in{\cal H}_0$,
 it follows that $T'=\varphi(X'_0)\neq \varphi(X_0)=T$,
 and this contradicts the
 assumption that $T$ is characterized from its moments.
 \end{proof}

 \begin{proof}[Proof of Theorem \ref{th.6.3}]
 Without loss of generality assume that $X\in{\cal H}_0$
 (equivalently, $H\in{\cal H}_0$ where $H(y):=F_X^{-1}(1-e^{-y})$, $y>0$)
 has ERS $\{\rho_1,\rho_2,\rho_3,\ldots\}=\{0,1,\rho_3,\ldots\}$.

 Suppose that (ii) is satisfied. Then, we can find $t_0\in(0,1)$ such that
 $G_{\rho}(t_0)=t_0+\rho_3 t_0^2+\cdots<\infty$; note that $\rho_n>0$
 for all $n\geq 2$. Moreover,
 \[
  w_n=\E |R_n|=\int_{0}^{\infty} \frac{y^{n-1}}{(n-1)!} e^{-y}\big|H(y)\big|dy
  = \rho_n + 2 \int_{0}^{m} \frac{y^{n-1}}{(n-1)!} e^{-y}\big|H(y)\big|dy,
 \]
 where $m:=\sup\{y>0:H(y)\leq 0\}\in(0,\infty)$, because $H$ is non-decreasing
 and $\int_{0}^\infty e^{-y}H(y)dy=0$. This verifies that
 $G_w(t_0)<\infty$, since $y\to e^{-y}H(y)\in L^1(0,\infty)$,
 $y\to e^{t_0 y}$ is bounded
 in $[0,m]$ and
 \[
 G_{w}(t_0)= G_\rho(t_0)+2 \int_{0}^m e^{t_0y}e^{-y}\big|H(y)\big|dy.
 \]
 It follows that
 $\int_{0}^\infty e^{t_0x}e^{-y}\big|H(x)\big|dx=G_{w}(t_0)<\infty$,
 and thus,
 \[
 \lim_{y\to\infty} \beta(y):=\lim_{y\to\infty} \int_{y}^\infty
 e^{t_0 x}e^{-x}\big|H(x)\big|dx=0.
 \]
 Since $H(x)$
 is non-decreasing and non-negative in $(m,\infty)$, we obtain
 \[
 \beta(y)\geq H(y) \int_{y}^\infty e^{-(1-t_0)x} dx
 =H(y) \frac{e^{-(1-t_0)y}}{1-t_0}, \ \ y>m,
 \]
 and hence, $H(y)=|H(y)|\leq C e^{(1-t_0)y}$, $y>m$,
 where $C=(1-t_0)\beta(m)<\infty$.
 It follows that
 \[
 \int_{m}^\infty e^{-y} \big|H(y)\big|^{1+\delta} dy \leq
 C^{1+\delta} \int_{m}^\infty \exp\Big[-y\Big(1-(1-t_0)(1+\delta)\Big)\Big] dy
 <\infty
 \]
 for $0\leq \delta<t_0/(1-t_0)$. Since
 $\E (X^+)^p=\int_{m}^\infty e^{-y} |H(y)|^p dy$, (i) is deduced.
 Fix now $\delta_0\in (0,t_0/(1-t_0))$ and $a>0$.
 Then, in view of \eqref{6},
 we have
 \[
 M_T(a)
 =\sum_{n=0}^\infty \frac{a^n}{n!} \E T^n
 =\sum_{n=0}^\infty \frac{a^n}{n!} (n+1)! (\rho_{n+2}-\rho_{n+1})
  \leq\sum_{n=0}^\infty (n+1)(w_{n+2}+w_{n+1})a^n.
 \]
 The last sum equals to
 \[
 \int_{0}^\infty \left(\sum_{n=0}^{\infty}
 \frac{(y+n+1)(ay)^n}{n!}\right)e^{-y} \big|H(y)\big|  dy
 =
 \int_{0}^\infty (y+1+ay)e^{ay}e^{-y} \big|H(y)\big|  dy.
 \]
 Since the function $y\to (y+1+ay)e^{ay}$ is bounded for
 $y\in[0,m]$, and $y\to e^{-y}H(y)\in L^1{(0,\infty)}$,
 the integral $\int_{0}^m (y+1+ay)e^{ay}e^{-y} |H(y)|  dy$
 is finite for all $a$. Furthermore,
 \[
 I:=\int_{m}^\infty (y+1+ay)e^{ay}e^{-y} \big|H(y)\big| dy
 =
 \int_{m}^\infty \Big(e^{-y/(1+\delta_0)} \big|H(y)\big|\Big)
 \Big((y+1+ay)e^{-\kappa y}\Big) dy,
 \]
 where $\kappa:=\delta_0/(1+\delta_0)-a>0$ if $a\in(0,\delta_0/(1+\delta_0))$.
 An application
 of H\"{o}lder's inequality (with $p=1+\delta_0$, $q=1+1/\delta_0$)
 to the last integral yields
 \[
 I\leq
 \left(\int_{m}^\infty e^{-y}
 \big|H(y)\big|^{1+\delta_0} dy \right)^{1/(1+\delta_0)}
 \left(\int_{m}^\infty (y+1+ay)^{1+1/\delta_0}
 e^{-\lambda y} dy\right)^{\delta_0/(1+\delta_0)},
 \]
 where $\lambda:=(1+1/\delta_0)\kappa>0$. Hence, $I$ is finite
 and, consequently, $M_T(a)<\infty$, for any $a<\delta_0/(1+\delta_0)$.

 So far, we have shown that (ii)$\Rightarrow$(iii)$\Rightarrow$(i)$\Rightarrow$(iv).
 The remaining implication, (iv)$\Rightarrow$(ii),
 is a by-product of Theorem \ref{th.1}. Indeed, if
 $M_T(a)=\E e^{aT}$
 is finite
 for some $a\in(0,1)$, then \eqref{6} shows that
 \[
 G_{\rho}(a)
 =\sum_{n\geq 2} a^{n-1}\sum_{j=0}^{n-2}\frac{\E T^j}{(j+1)!}
 =\sum_{j\geq 0} \frac{\E T^j}{(j+1)!} \sum_{n\geq {j+2}} a^{n-1}
 =\frac{1}{1-a}\E\Big[\frac{e^{a T}-1}{T}\Big].
 \]
 Since $(e^{a t}-1)/t<a e^{at}$ for $t>0$, (ii) is proved.
 Finally, from the preceding  calculation,
 \[
 G_{\rho}(t)=
 \frac{1}{1-t}\E \int_{0}^t e^{sT} ds=
 \frac{1}{1-t}\int_{0}^t \E e^{sT} ds,
 \ \ \ 0<t\leq a,
 \]
 and \eqref{6.8} is deduced.

 Obviously, the results extend to a
 $t$-interval $-\epsilon_0<t<\epsilon_0$ by analytic continuation.
 \end{proof}


 \section{Construction of $X_0\in{\cal H}_0$ from $T\in{\cal T}$}
 \label{app.d}
 \setcounter{equation}{0}

 We shall provide a detailed proof of Theorem \ref{th1}(ii), which also
 verifies the half counterpart of Proposition \ref{prop.phi}, showing
 that the mapping $\varphi'$ is well-defined with
 domain ${\cal T}$ and range into ${\cal H}_0$, as stated. We notice
 that the present appendix is self-contained; it does not require any
 further results from the present article.

 Suppose we are given an r.v.\ $T\in{\cal T}$ with d.f.\ $F_T$, i.e.,
 $F_T(0)=0$ and $\E T^n<\infty$ for all $n$.
 Define
 \be
 \label{H}
 H(y):=H_0(y)+c_T-e F_T(1-), \ \ y>0,
 \ee
 where $H_0$ is as in \eqref{6.1} and $c_T$ as in \eqref{cT},
 and rewrite \eqref{6.1} as
 \be
 \label{4.2}
 H(y)=\left\{
 \begin{array}{l}
 -e\big[F_T(1-)-F_T(y-)\big]-\int_{y}^{1} \frac{1-x}{x^2}
 e^x \big[F_T(x)-F_T(y-)\big]  dx,
 \\
 \hfill 0<y\leq 1,
 \\
 \vspace{-2ex}
 \\
 e\big[F_T(y-)-F_T(1-)\big]+\int_{1}^{y} \frac{x-1}{x^2}
 e^x \big[F_T(y-)-F_T(x)\big] dx,
 \\
 \hfill 1\leq y<\infty.
 \end{array}
 \right.
 \ee
 From \eqref{4.2} we see that
 $H(y)\leq 0$ for $y\leq 1$ and $\geq 0$ otherwise.

 \begin{lem}
 \label{lem.1}
 $H$ is non-decreasing and left-continuous.
 \end{lem}

 \begin{proof}
 Left-continuity is obvious. Also, $H$ is non-positive in $(0,1]$
 and non-negative in $[1,\infty)$. Choose now $y_1,y_2$ with
 $0<y_1<y_2\leq 1$.
 Then,
 \be
 \label{mono}
 H(y_2)-H(y_1)=\frac{e^{y_2}}{y_2}\Pr(y_1\leq T<y_2)+
 \int_{y_1}^{y_2} \frac{1-x}{x^2}e^x \Pr(y_1\leq T\leq x) dx \geq 0.
 \ee
 A similar argument applies to the case $1\leq y_1<y_2<\infty$.
 \end{proof}

 \begin{lem}
 \label{lem.2} {\rm (i)} For all $x\geq 0$,
 \be
 \label{m}
 \int_{0}^{x} e^{-y}\Big[F_T(x)-F_T(y)\Big] dy
 =\int_{(0,x]} (1-e^{-t})  d F_T(t).
 \ee

 \noindent
 {\rm (ii)}
 The function $y\to e^{-y}H(y)\in L^{1}(0,m)$ for any finite $m>0$.
 \end{lem}
 \begin{proof}
 Consider the non-negative random variable
 $Y:=(1-e^{-T}) I(T\leq x)=g(T)$, for
 which it is easily verified that
 $\Pr(Y>u)=F_T(x)-F_T(-\log(1-u))$ for $0\leq u<1-e^{-x}$, and $\Pr(Y>u)=0$ for
 $u\geq 1-e^{-x}$. Then, we can compute the expectation of $Y$ by means of two
 different integrals, namely,
 \[
 \E Y = \int_{0}^{1-e^{-x}} \Big[F_T(x)-F_T(-\log(1-u))\Big] du, \ \
 \ \
 \E
 g(T)
 = \int_{(0,x]} (1-e^{-t})  d F_T(t).
 \]
 The integrals above are equal, and the substitution $-\log(1-u)=y$
 yields \eqref{m}.
 Fix now $\delta\in(0,1)$.
 From \eqref{4.2} we have
 \[
 \int_0^{\delta} e^{-y}|H(y)|  dy\leq e + \int_{0}^\delta
 e^{-y}\int_{y}^{1} \frac{1-x}{x^2}
 e^x \Big[F_T(x)-F_T(y)\Big]dxdy=e+I, \ \ \mbox{say},
 \]
 noting that
 \[
 e^{-y}\frac{1-x}{x^2}
 e^x \Big[F_T(x)-F_T(y-)\Big]
 =
 e^{-y}\frac{1-x}{x^2}
 e^x \Big[F_T(x)-F_T(y)\Big]
 \]
 for almost all  $(y,x)\in (0,\infty)\times(0,1)$.
 Interchanging the order of integration
 according to Tonelli's
 theorem, we get
 \begin{eqnarray*}
 I & = & \int_{0}^{\delta} \frac{1-x}{x^2} e^x \int_{0}^{x}
 e^{-y} \Big[F_T(x)-F_T(y)\Big]  dy dx \\
 && +
 \int_{\delta}^1 \frac{1-x}{x^2} e^x \int_{0}^{\delta}
 e^{-y} \Big[F_T(x)-F_T(y)\Big]  dy dx=I_1+I_2.
 \end{eqnarray*}
 Obviously, $I_2$ is finite and it remains to verify that $I_1<\infty$.
 In view of  \eqref{m}
 we obtain
 \[
 I_1=\int_{0}^{\delta}
 \frac{1-x}{x^2} e^x \int_{(0,x]} (1-e^{-t}) d F_T(t)dx
 = \int_{(0,\delta]}
 (1-e^{-t})
 \int_{t}^{\delta} \frac{1-x}{x^2} e^x dx d F_T(t).
 \]
 The last equation shows that $I_1$ is finite, because the inner
 integral is less than $e^t/t$.
 Thus,
 \[
 I_1\leq \int_{(0,\delta]}
 \frac{e^t-1}{t}  d F_T(t) = \E\Big[\frac{e^T-1}{T}I(T\leq \delta)\Big]
 \]
 and the function $T\to (e^T-1)I(T\leq \delta)/T$ is (non-negative and)
 bounded. Finally, since $|H|$ is bounded in $[\delta, m]$
 (see Lemma \ref{lem.1}), the lemma is proved.
 \end{proof}

 \begin{lem}
 \label{lem.3}
 Define
 \be
 \label{4.3}
 a_k(t):=\int_{t}^{\infty}u^k e^{-u}  du=
 k! e^{-t}\sum_{j=0}^{k}\frac{t^j}{j!},  \ \ \ \
 \ t\geq0, \ \ k=0,1,\ldots \ .
 \ee
 Then,
 \be
 \label{4.4}
 \int_x^\infty y^k e^{-y} \Big[F_T(y)-F_T(x)\Big] dy =
 \int_{(x,\infty)}a_k(t) dF_T(t),  \ \ x\geq0, \ \ k=0,1,\ldots \ .
 \ee
 \end{lem}
 \begin{proof}
 $a_k$ is strictly decreasing
 with $a_k(0)=k!$ and $a_k(\infty)=0$. Fix $x\geq 0$ and consider the
 bounded
 non-negative
 r.v.\ $Y:=a_k(T)I(T> x)$. Then, $\Pr(Y> y)=F_T(a_k^{-1}(y)-)-F_T(x)$
 for $0\leq y<a_k(x)$, and $\Pr(Y>y)=0$ for $y\geq a_k(x)$, where
 $a_k^{-1}$ is the (usual) inverse function of $a_k$.
 Since $F_T(a_k^{-1}(y)-)=F_T(a_k^{-1}(y))$ for almost all $y\in (0,\infty)$,
 we obtain
 \[
 \E Y = \int_{0}^{a_k(x)} \Big[F_T(a_k^{-1}(y))-F_T(x)\Big]  dy
 = \int_{x}^{\infty} t^k e^{-t} \Big[F_T(t)-F_T(x)\Big]  dt,
 \]
 where we made use of the substitution $t=a_k^{-1}(y)$. On the other hand,
 \[
 \E Y=\E\Big[ a_k(T) I(T>x)\Big]=\int_{(x,\infty)} a_k(t)  dF_T(t),
 \]
 and \eqref{4.4} is proved.
 \end{proof}

 \begin{lem}
 \label{lem.4}
 $\int_0^\infty y^k e^{-y}|H(y)|  dy<\infty$ for $k=0,1,\ldots$ \ .
 \end{lem}
 \begin{proof}
 Fix $k\in\{0,1,\ldots\}$, $m\in (1,\infty)$, and write
 \[
 \int_0^\infty y^k e^{-y}|H(y)| dy = \int_{0}^{m}
 +\int_{m}^{\infty} y^k e^{-y}|H(y)| dy =I_1+I_2.
 \]
 Lemma \ref{lem.2}(ii)
 shows that
 $I_1$ is finite,
 and we proceed to verify that $I_2$ is also finite.
 Using \eqref{4.2} we have
 \[
 I_2\leq e k! +
 \int_{m}^\infty \int_{1}^{y} y^k e^{-y}
 \frac{x-1}{x^2} e^x \Big[F_T(y)-F_T(x)\Big] dx dy=e k! +I_3,
 \]
 noting that
 \[
 y^k e^{-y}\frac{x-1}{x^2} e^x \Big[F_T(y-)-F_T(x)\Big]
 =y^k e^{-y} \frac{x-1}{x^2} e^x \Big[F_T(y)-F_T(x)\Big]
 \]
 for almost all $(y,x)\in(m,\infty)\times(1,\infty)$.
 It remains to show $I_3<\infty$.
 Using Tonelli's
 theorem,
 \begin{eqnarray*}
 I_3
 &=&
 \int_1^{m} \frac{x-1}{x^2} e^x
 \int_{m}^{\infty} y^k e^{-y}
 \Big[F_T(y)-F_T(x)\Big] dy dx
 \\
 &&
 +
 \int_{m}^\infty \frac{x-1}{x^2} e^x
 \int_{x}^{\infty} y^k e^{-y}
 \Big[F_T(y)-F_T(x)\Big] dy dx=J_1+J_2.
 \end{eqnarray*}
 Now $J_1$ is obviously finite, because the inner integral
 is less that $k!$ and the function $x\to (x-1)e^x/x^2$ is bounded
 for $x\in[1,m]$. Applying Lemma \ref{lem.3} to the inner integral
 in $J_2$ we obtain
 \[
 J_2=\int_{m}^\infty \frac{x-1}{x^2} e^x
 \int_{(x,\infty)} a_k(t) dF_T(t) dx
 =\int_{(m,\infty)} a_k(t) \int_{m}^t
 \frac{x-1}{x^2} e^x dx dF_T(t).
 \]
 Therefore, since the inner integral is less than $e^t/t$,
 and
 \[
 \frac{e^t}{t} a_k(t)=k!\left(
 \frac{1}{t}+\sum_{j=0}^{k-1} \frac{t^j}{(j+1)!}\right)\leq
 \frac{k!}{m}+k!\sum_{j=0}^{k-1} \frac{t^j}{(j+1)!}, \ \ \ t> m,
 \]
 see \eqref{4.3}, we arrive at the inequality
 $J_2
 \leq k!/m
 +
 k! \sum_{j=0}^{k-1} \E T^j/(j+1)!$,
 and this is finite because $T$ has been assumed to
 possess finite moments of any order.
 \end{proof}

 From Lemmas \ref{lem.1}, \ref{lem.4} we conclude that
 $H\in {\cal H}$, so that $H_0\in {\cal H}$
 (since these functions differ by a constant--see \eqref{H}).
 We now proceed to show that $H_0\in{\cal H}_0$.

 \begin{lem}
 \label{lem.5}
 For each $k=1,2,\ldots$,
 \be
 \label{4.5}
 \int_0^\infty (y^k-k!)e^{-y} F_T(y) dy = \E \Big[a_k(T)-k! e^{-T}\Big],
 \ee
 where $a_k$ is given by \eqref{4.3}.
 \end{lem}
 \begin{proof}
 Substitute $x=0$ in Lemma \ref{lem.3} and observe that $a_k(0)=k!$,
 $a_0(t)=e^{-t}$ and $F_T(0)=0$.
 \end{proof}

 \begin{lem}
 \label{lem.6}
 For each $k=1,2,\ldots$,
 \be
 \label{4.6}
 \int_0^\infty (y^k-k!)e^{-y}H(y) dy = k! \sum_{j=0}^{k-1} \frac{\E T^j}{(j+1)!}.
 \ee
 \end{lem}
 \begin{proof}
 Set $\beta_k:=k!^{1/k}$ so that $1=\beta_1<\beta_2<\cdots\to\infty$, as
 $k\to\infty$, and note that the integral in \eqref{4.6} is finite --
 see Lemma \ref{lem.4}. Clearly, $y^k>k!$ for $y\in (\beta_k,\infty)$ and
 $y^k<k!$ for $y\in(0,\beta_k)$. We split the integral
 in \eqref{4.6}  as follows:
 \begin{eqnarray*}
  \int_0^{1} (k!-y^k)e^{-y}(-H(y)) dy-
 \int_{1}^{\beta_k} (k!-y^k) e^{-y} H(y) dy \\
 + \int_{\beta_k}^\infty (y^k-k!)e^{-y}H(y) dy
 =I_1-I_2+I_3.
 \end{eqnarray*}
 Now we compute these three integrals. From \eqref{4.2},
 \begin{eqnarray*}
 I_1
 &=& e \int_0^{1} (k!-y^k)e^{-y}
 \Big[F_T(1-)-F_T(y)\Big] dy
 \\
 &&+
 \int_0^{1}(k!-y^k)e^{-y}\int_{y}^{1} \frac{1-x}{x^2}
 e^x \Big[F_T(x)-F_T(y)\Big] dx dy=eI_{11}+I_{12}.
 \end{eqnarray*}
 Similarly,
 \begin{eqnarray*}
 I_2
 &=& e \int_1^{\beta_k} (k!-y^k)e^{-y}
 \Big[F_T(y)-F_T(1-)\Big] dy
 \\
 &&+
 \int_1^{\beta_k}(k!-y^k)e^{-y}\int_1^{y} \frac{x-1}{x^2}
 e^x \Big[F_T(y)-F_T(x)\Big] dx dy=eI_{21}+I_{22},
 \end{eqnarray*}
 and, finally,
 \begin{eqnarray*}
 I_3
 &=& e \int_{\beta_k}^\infty (y^k-k!)e^{-y}
 \Big[F_T(y)-F_T(1-)\Big] dy
 \\
 &&+
 \int_{\beta_k}^\infty (y^k-k!)e^{-y}\int_1^{y} \frac{x-1}{x^2}
 e^x \Big[F_T(y)-F_T(x)\Big] dx dy=eI_{31}+I_{32}.
 \end{eqnarray*}
 The above calculation shows that
 \[
 e (I_{11}-I_{21}+I_{31})=
 eF_T(1-) \int_{0}^{\infty} (k!-y^k) e^{-y}  dy
 +e \int_0^\infty (y^k-k!)e^{-y} F_T(y) dy.
 \]
 Similarly,
 \begin{eqnarray*}
 I_{12}-I_{22}+I_{32}
 &=&
 \int_0^{1}(k!-y^k)e^{-y}
 \int_{y}^{1} \frac{1-x}{x^2}
 e^x \Big[F_T(x)-F_T(y)\Big]  dx dy
 \\
 &&
 +
 \int_{1}^\infty (y^k-k!)e^{-y}\int_1^{y} \frac{x-1}{x^2}
 e^x \Big[F_T(y)-F_T(x)\Big]  dx dy.
 \end{eqnarray*}
 Observing that $\int_{0}^{\infty} (k!-y^k) e^{-y}  dy=0$, we
 finally obtain
 \be
 \label{4.7}
 \int_0^\infty (y^k-k!) e^{-y}H(y) dy =e J_1+J_2+J_3,
 \ee
 where
 \begin{eqnarray*}
 J_1&=& \int_0^\infty (y^k-k!)e^{-y}
 F_T(y) dy,
 \\
 J_2&=& \int_0^{1}
 \int_{y}^{1} (k!-y^k)e^{-y}\frac{1-x}{x^2}
 e^x \Big[F_T(x)-F_T(y)\Big] dx dy,
 \\
 J_3&=&
 \int_{1}^\infty \int_1^{y} (y^k-k!)e^{-y}\frac{x-1}{x^2}
 e^x \Big[F_T(y)-F_T(x)\Big] dx dy.
 \end{eqnarray*}
 The integrand in $J_2$ is non-negative, so we can change
 the order of integration.
 In order to justify that this is also permitted for $J_3$,
 we compute
 \begin{eqnarray*}
 \int_{1}^\infty \int_1^{y} \left|(y^k-k!)e^{-y}\frac{x-1}{x^2}
 e^x \Big[F_T(y)-F_T(x)\Big]\right|  dx dy
 \ \ \ \ \ \ \ \ \ \ \ \ \ \ \ \ \ \ \ \ \ \ \ \
 \ \ \ \ \ \ \ \ \ \ \ \ \ \ \ \ \ \ \ \ \ \ \ \
 \mbox{}
 \\
 \leq
 \int_{1}^\infty \frac{x-1}{x^2}
 e^x \int_x^{\infty} (y^k+k!)e^{-y} \Big[F_T(y)-F_T(x)\Big] dy dx
 \ \ \ \ \ \ \ \ \ \ \ \ \ \ \ \
 \ \ \ \ \ \ \ \ \ \ \ \ \ \ \ \ \ \ \ \ \ \ \ \
 \mbox{}
 \\
 =
 \int_{1}^\infty \frac{x-1}{x^2}
 e^x \int_{(x,\infty)} \Big[a_k(t)+k! a_0(t)\Big] dF_T(t) dx
 \ \ \ \ \ \ \ \ \ \ \ \ \mbox{(Lemma \ref{lem.3})}
 \\
 =
 \int_{(1,\infty)} \Big[a_k(t)+k! a_0(t)\Big]
 \left(\frac{e^t}{t}-e\right)
 \
 dF_T(t)
 \ \ \ \ \ \ \ \ \ \ \ \ \ \ \ \ \
 \ \ \ \ \ \ \ \ \ \ \ \
 \mbox{(Tonelli)}
 \\
 \leq
 \int_{(1,\infty)} \Big[a_k(t)+k! a_0(t)\Big]
 \frac{e^t}{t}
 dF_T(t)<\infty,
 \ \ \ \ \ \ \ \ \ \ \ \ \ \ \ \ \
 \ \ \ \ \ \ \ \ \ \ \ \
  \ \ \ \ \ \ \ \ \ \ \
   \ \ \
 \mbox{}
 \end{eqnarray*}
 because, for $t>1$,
 \[
 \Big[a_k(t)+k! a_0(t)\Big]
 \frac{e^t}{t}=\frac{2 k!}{t}+k! \sum_{j=0}^{k-1} \frac{t^{j}}{(j+1)!}
 \leq 2 k! + k! \sum_{j=0}^{k-1} \frac{t^{j}}{(j+1)!},
 \]
 and $T$ has finite moments of any order. Thus,
 \begin{eqnarray*}
 J_2&=& \int_0^{1}
 \frac{1-x}{x^2}
 e^x  \int_{0}^{x}(k!-y^k)e^{-y} \Big[F_T(x)-F_T(y)\Big] dy dx,
 \\
 J_3&=&
 \int_{1}^\infty \frac{x-1}{x^2}
 e^x \int_x^{\infty} (y^k-k!)e^{-y} \Big[F_T(y)-F_T(x)\Big] dy dx.
 \end{eqnarray*}
 The inner integral in $J_3$
 equals to
 $
 \int_{(x,\infty)} \big[a_k(t)-k!e^{-t}\big] dF_T(t)
 $; see Lemma \ref{lem.3}.
 Since $a_k(t)-k! e^{-t}=k! e^{-t} \sum_{j=1}^k t^j/j!$,
 we obtain (after changing the order of integration once again)
 \begin{eqnarray}
 J_3&=& k! \int_{(1,\infty)} \left(e^{-t}\sum_{j=1}^k
 \frac{t^j}{j!}\right)\left(\frac{e^t}{t}-e\right) dF_T(t)
 \nonumber
 \\
 \label{4.8}
 &=& k! \int_{(1,\infty)} \sum_{j=0}^{k-1}
 \frac{t^j}{(j+1)!} dF_T(t) -e k!
 \int_{(1,\infty)} e^{-t}\sum_{j=1}^k
 \frac{t^j}{j!} dF_T(t).
 \end{eqnarray}
 Next, we make similar calculations for the inner integral in $J_2$.
 We have
 \begin{eqnarray*}
 &&
 \int_{0}^{x}(k!-y^k)e^{-y} \Big[F_T(x)-F_T(y)\Big] dy
 \\
 &&
 =\int_{0}^{\infty}(y^k-k!)e^{-y} \Big[F_T(y)-F_T(x)\Big] dy
 -\int_{x}^{\infty}(y^k-k!)e^{-y} \Big[F_T(y)-F_T(x)\Big] dy
 \\
 &&
 =\int_{0}^{\infty} (y^k-k!) e^{-y} F_T(y) dy
 -\int_{(x,\infty)} \Big[a_k(t)-k! e^{-t}\Big] dF_T(t)
 \\
 &&
 =\E\Big[a_k(T)-k! e^{-T}\Big]-\int_{(x,\infty)}
 \Big[a_k(t)-k! e^{-t}\Big] dF_T(t)
 \\
 &&
 =
 \int_{(0,x]}
 \Big[a_k(t)-k! e^{-t}\Big] dF_T(t),
 \end{eqnarray*}
 where we made use of Lemmas \ref{lem.3} and \ref{lem.5}
 and the fact that $\int_{0}^{\infty}(y^k-k!)e^{-y}  dy
 =0=a_k(0)-k! e^{-0}$. Therefore,
 \begin{eqnarray}
 J_2&=&-e\int_{(0,1]} \Big[a_k(t)-k! e^{-t}\Big] dF_T(t)+
 \int_{(0,1]} \frac{e^t}{t}\Big[a_k(t)-k! e^{-t}\Big] dF_T(t)
 \nonumber
 \\
 &=&
 \label{4.9}
 -e k! \int_{(0,1]} e^{-t} \sum_{j=1}^k \frac{t^j}{j!}\  dF_T(t)+
 k! \int_{(0,1]} \sum_{j=0}^{k-1} \frac{t^j}{(j+1)!} dF_T(t).
 \end{eqnarray}
 Finally, from Lemma \ref{lem.5},
 \be
 \label{4.10}
 e J_1=e \int_{(0,\infty)} \Big[a_k(t)-k! e^{-t}\Big]  dF_T(t)
 =e k! \int_{(0,\infty)} e^{-t} \sum_{j=1}^k \frac{t^j}{j!}  dF_T(t).
 \ee
 Combining \eqref{4.8}--\eqref{4.10} we obtain
 \begin{eqnarray*}
 e J_1+J_2+J_3=k! \int_{(0,1]} \sum_{j=0}^{k-1} \frac{t^j}{(j+1)!} dF_T(t)
 +
 k! \int_{(1,\infty)} \sum_{j=0}^{k-1}
 \frac{t^j}{(j+1)!} dF_T(t),
 \end{eqnarray*}
 and from \eqref{4.7} we conclude \eqref{4.6}.
 \end{proof}

 \begin{lem}
 \label{lem.cT}
 $\int_0^{\infty} e^{-y}H(y) dy=c_T-e F_T(1-)$, where $c_T$ is as in
 \eqref{cT} and $H$ as in \eqref{H}.
 \end{lem}
 \begin{proof} Using \eqref{4.2} and applying Lemma \ref{lem.3} (for $k=0$),
 we obtain
 \begin{align*}
 \int_1^{\infty} e^{-y}H(y) dy
 =&\Pr(T=1)+e\int_{(1,\infty)}e^{-t}dF_T(t)\\
 & +\int_{1}^{\infty}
  \frac{x-1}{x^2}e^x\int_{(x,\infty)}e^{-t}dF_T(t) dx
  =\Pr(T=1)+\E\Big[ \frac{1}{T}I(T>1)\Big].
 \end{align*}
 Similarly, from Lemma \ref{lem.2}(i), we get
 \begin{align*}
 \int_0^{1} e^{-y}(-H(y)) dy
 =&-(e-1)\Pr(T=1)+e\int_{(0,1]}(1-e^{-t})dF_T(t)\\
 & +\int_{0}^{1}
  \frac{1-x}{x^2}e^x\int_{(0,x]}(1-e^{-t})dF_T(t) dx
 \\
 =&-(e-1)\Pr(T=1)+\E\Big[ \frac{e^T-1}{T}I(T\leq 1)\Big].
 \end{align*}
 Subtracting the above equations we deduce the desired result.
 \end{proof}

 \noindent
 \begin{proof}[{Proof of Theorem} \ref{th1}{\rm (ii)} and of the second
 part of Proposition \ref{prop.phi}]
 Let ${E}$ be a standard exponential r.v.,
 and set $X_0:=H_0({E})$, where $H_0$ is given by \eqref{6.1}.
 From Lemmas \ref{lem.cT}, \ref{lem.1}, \ref{lem.4} and
 \ref{lem.6} (with $k=1$), and in view of \eqref{E},
 we see that $H_0\in{\cal H}_0$, i.e.,
 $\rho_1=0$, $\rho_2=1$, where $\{\rho_n\}_{n\geq 1}$ is the ERS
 from $X_0$. Noting that $X_0=\varphi'(T)$, see Definition \ref{def.6.4},
 we have proven that the mapping $\varphi'$ is well-defined for all
 $T\in{\cal T}$, and its values are in ${\cal H}_0$.
 Finally, from Lemma \ref{lem.6},
 \[
 \rho_{n+1}=
 \rho_{n+1}-\rho_1=\int_0^\infty \left(\frac{y^n}{n!}-1\right)e^{-y} H_0(y) dy
 = \sum_{j=0}^{n-1} \frac{\E T^j}{(j+1)!}, \ \ n=1,2,\ldots
 \]
 (the last equality is justified because $H_0-H=c$, constant),
 completing the proof.
 \end{proof}





\begin{thebibliography}{99}


 \bibitem{Ahs}
 {\sc Ahsanullah, M.} (1995).
 \textit{Record Statistics}. Nova
 Science Publishers, N.Y.

 \bibitem{Aliev98}
 {\sc Aliev, F.A.} (1998). Characterizations
 of distributions through weak records.
 {\it J.\ Appl.\ Stat.\ Science},
 {\bf 8}, 13--16.

 \bibitem{ABN1998}
 {\sc
 Arnold, B.C.; Balakrishnan, N.; Nagaraja, H.N.} (1998).
 \textit{Records}. Wiley, N.Y.


 \bibitem{BNK2019}
 {\sc
 Barakat, H.M.;
 Nigm, E.M.; Khaled, O.M.} (2019).
 {\it  Statistical Techniques for Modelling Extreme Value Data
 and Related Applications}. Cambridge
 Scholars Publishing, U.K.


 \bibitem{DD07}
 {\sc Danielak, K.; Dembi\'{n}ska, A.} (2007).
 On characterizing discrete distributions via conditional expectations of
 weak record values.
 {\it Metrika}, {\bf 66}(2), 129--138.


 \bibitem{DW00}
 {\sc Dembi\'{n}ska, A.; Wesolowski, J.} (2000).
 Linearity of regression
 for non-adjacent record values. {\it J.\ Statist.\ Plann.\ Inference},
 {\bf 90}, 195--205.

 \bibitem
  {EKM1997}
    {\sc Embrechts, P.; Kl\"{u}ppelberg, C; Mikosch, T.}\ (1997).
    \textit{Modelling Extremal Events.}
    Springer-Verlag, Berlin.



 \bibitem
 {Hausdorff1921}
    {\sc Hausdorff, F.}\ (1921).
    Summationmethoden und momentfolgen. I.
    \textit{Math.\ Zeitchrift},  {\bf9}(1), 74--109.

 \bibitem
 {HillSpruill1994}
    {\sc Hill, T.P.; Spruill, M.C.}\ (1994).
    On the relationship between convergence in
    distribution and convergence
    of expected extremes.
    \textit{Proc.\ Amer.\ Math.\ Soc.}, {\bf121}(4), 1235--1243.

  \bibitem
 {HillSpruill2000}
    {\sc Hill, T.P.; Spruill, M.C.}\ (2000).
    Erratum to ``On the relationship between convergence in
    distribution and convergence
    of expected extremes".
    \textit{Proc.\ Amer.\ Math.\ Soc.}, {\bf128}, 625--626.

  \bibitem
 {Hoeffding1953}
    {\sc Hoeffding, W.}\ (1953).
    On the distribution of the expected values of the order statistics.
    \textit{Ann.\ Math.\ Statist.},  {\bf24}(1), 93--100.

  \bibitem
 {Huang1998}
    {\sc Huang, J.S.}\ (1998).
    Sequences of expectations
    of maximum-order statistics.
    \textit{Statist.\ Probab.\ Lett.}, {\bf38}, 117--123.


 \bibitem
 {Korwar1984}
  {\sc Korwar, R.M.} (1984). On characterizing distributions
  for which the second record value has a linear regression
  on the first. {\it Sankhy\={a} B}, {\bf 46}, 108--109.

 \bibitem
 {Kadane1971}
    {\sc Kadane, J.B.}\ (1971).
    A moment problem for order statistics.
    \textit{Ann.\ Math.\ Statist.},  {\bf42},
    745--751.

 \bibitem
 {Kadane1974}
    {\sc Kadane, J.B.}\ (1974).
    A characterization of triangular arrays which are
    expectations of order statistics.
    \textit{J.\ Appl.\ Probab.},  {\bf11},
    413--416.

 \bibitem{KiBe84}
 {\sc Kirmani, S.N.U.A.; Beg, M.I.} (1984).
 On characterization of distributions
 by expected records. {\it Sankhy\={a} A}, {\bf 46}, 463--465.

 \bibitem
 {Kolo2000}
    {\sc Kolodynski, S.}\ (2000).
    A note on the sequence of expected extremes.
    \textit{Statist.\ Probab.\ Lett.}, {\bf47}, 295--300.


 \bibitem{Lin87}
 {\sc Lin, G.D.} (1987).
 On characterizations of distributions via moments of record values.
 {\it Probab.\ Th.\ Rel.\ Fields}, {\bf 74}, 479--483.

 \bibitem{LiHu87}
 {\sc
 Lin, G.D.; Huang, J.S.} (1987). A note on the sequence
 of expectations of maxima and of record values.
 {\it Sankhy\={a} A},
 {\bf 3}, 272--273.

 \bibitem{LiSto16}
 {\sc
 Lin, G.D.; Stoyanov, J.} (2016). On the moment determinacy of products
 of non-identically distributed random variables. {\it Probab.\ Math.\
 Statist.},
 {\bf 36}(1), 21--33.

 \bibitem{LoWe01}
 {\sc
 Lopez-Blazquez, F.; Wesolowski, J.} (2001). Discrete distributions
 for which the regression of the first record on the second is
 linear. {\it Test}, {\bf 10}, 121--131.



 \bibitem
 {Nag77}
  {\sc
  Nagaraja, H.N.} (1977). On a characterization based
  on record values.
 {\it Austral.\ J.\ Statist.}, {\bf 19}, 70--73.

 \bibitem
 {Nag78}
  {\sc
  Nagaraja, H.N.} (1978). On the expected values of record values.
 {\it Austral.\ J.\ Statist.}, {\bf 20}, 176--182.

 \bibitem
 {Nag88}
  {\sc
  Nagaraja, H.N.} (1988). Some characterizations of continuous
  distributions based
  on regressions of adjacent order statistics and record values.
 {\it Sankhy\={a} A}, {\bf 50}, 70--73.

 \bibitem
 {Nev01}
    {\sc
  Nevzorov, V.B.} (2001).
  \textit{Records: Mathematical Theory}. American Mathematical Society,
  Providence, RI.

 \bibitem
 {Pap01}
    {\sc Papadatos, N.} (2001).
  Distribution and expectation bounds on order statistics
  from possibly dependent variates.
  {\it Statistics and Probability Letters}, {\bf 54}, 21--31.


 \bibitem{Papad12}
 {\sc Papadatos, N.} (2012). Linear estimation of location
 and scale parameters using partial maxima.
 {\it Metrika},  {\bf 75}, 243--270.

 \bibitem{Papad17}
 {\sc Papadatos, N.} (2017). On sequences
 of expected maxima and expected ranges.
 {\it J.\ Appl.\ Probab.},  {\bf 54}, 1144--1166.

 \bibitem
 {Raq02}
  {\sc
 Raqab M.Z.} (2002).
 Characterizations of distributions based on the conditional
 expectations of record values.
 {\it Statistics \& Decisions.}\
 {\bf 20}, 309--319.

 \bibitem
 {Res73}
  {\sc
 Resnick, S.I.} (1973). Record Values and Maxima.
 {\it Ann.\ Probab.}, {\bf 1}(4), 650--662.

 \bibitem
 {Res87}
  {\sc
 Resnick, S.I.} (1987). \textit{Extreme values, regular
 variation, and point processes}. Springer, N.Y.


 \bibitem
 {Step93}
  {\sc
 Stepanov, A.V.} (1993). A characterization theorem for weak
 records.
 {\it Theory Probab.\ Appl.}, {\bf 38}, 762--764.

 \bibitem
 {Sto13}
  {\sc
  Stoyanov, J.} (2013).
  \textit{Counterexamples in Probability}. Dover Publications,
  N.Y.

 \bibitem
 {Sto13b}
  {\sc
 Stoyanov, J.} (2013). Hardy's condition in the moment
 problem for probability distributions.
 {\it Theory.\ Probab.\ Appl.} (SIAM), {\bf 57}(4),
 699–-708.

 \bibitem
 {StoTolm}
  {\sc
 Stoyanov, J.; Tolmatz, L.} (2005). Method for constructing Stieltjes
 classes for M-indeterminate probability distributions.
 {\it Appl.\ Math.\ Comput.}, {\bf 165}, 669–-685.

 \bibitem
 {TrBla85}
  {\sc Tryfos, P; Blackmore, R.}
  (1985). Forecasting records.
 {\it J.\ Am.\ Stat.\ Assoc.}, {\bf 80}, 46--50.



 \bibitem
 {Yanev12}
 {\sc
  Yanev, G.P.} (2012). Characterization of exponential distribution via
 regression of one record value on two non-adjacent record values.
 {\it Metrika}, {\bf 75}(6), 743--760.

\end{thebibliography}
\end{document}